\documentclass[12pt,reqno]{amsart} % use larger type; default would be 10pt

\usepackage[utf8]{inputenc} % set input encoding (not needed with XeLaTeX)
\usepackage{bbm}
\usepackage[labelfont=bf]{caption}

\usepackage{graphicx} % support the \includegraphics command and options

\usepackage{booktabs} % for much better looking tables
\usepackage{array} % for better arrays (eg matrices) in maths
\usepackage{paralist} % very flexible & customisable lists (eg. enumerate/itemize, etc.)
\usepackage{verbatim} % adds environment for commenting out blocks of text & for better verbatim
\usepackage{subfig} % make it possible to include more than one captioned figure/table in a single float
\usepackage{mathrsfs}
\usepackage{amssymb}
\usepackage{amsthm}
\usepackage{amsmath,amsfonts,amssymb,esint}
\usepackage{graphics,color}
\usepackage{graphicx}
\usepackage{enumerate}
\usepackage{mathtools,centernot}
\usepackage{cases}

\usepackage{float}
\usepackage{verbatim}

\newtheorem*{remark 1}{Remark}

\newtheorem{lemma}{Lemma}
\newtheorem{theorem}{Theorem}

\newtheorem{remark}{Remark}

\usepackage{ color, graphicx}
\usepackage{mathrsfs, dsfont}

\usepackage{hyperref}

\numberwithin{equation}{section}

\newcommand\nc{\newcommand}
\nc\hd{\widehat{D}}
\nc\kp{\kappa}
\nc\8{\infty}

\title{Linear Stability of the 2D Irrotational Circulation Flow around An Elliptical Cylinder}
\author{Xiao Ma}
\address{Princeton University}
\email{xiaom@math.princeton.edu}

\begin{document}
\begin{abstract}
In this article we prove a linear inviscid damping result with optimal decay rates of the 2D irrotational circulation flow around an elliptical cylinder. In our result, all components of the asymptotic velocity field do not vanish and the asymptotic flow lines are not ellipse any more. 
\end{abstract}

\maketitle

\section{Introduction}

In this paper, we study the linear stability of the 2D irrotational circulation flow around an elliptical cylinder. %(The flow is described by the following figure or equation (\ref{}))

An elliptical cylinder is the region $\mathcal{C}=\{(X,Y,Z): \frac{X^2}{A^2}+\frac{Y^2}{B^2}\le 1\}\subseteq \mathbb{R}^3$. In this paper we consider an irrotational flow that is independent of $Z$. This flow circulates around cylinder $\mathcal{C}$ in domain $\mathbb{R}^3\backslash \mathcal{C}$ and is the unique irrotational solution of the 2D Euler equation

\begin{equation}
\left\{\begin{array}{cc}
     \partial_{t} \omega_0+\nabla^{\perp}\psi_0\cdot\nabla\omega_0=0.  \\
     \omega_0=\Delta \psi_0=0
\end{array}\right.
\end{equation}
with boundary condition $\psi|_{\partial\mathcal{C}}=0$, $\psi|_{\infty}\sim a+b\ ln(X^2+Y^2)$. Here the Euler equation is written in the vorticity-stream function formulation. $\omega_0$ is the vorticity and $\psi_0$ is the stream function. $\nabla=\begin{bmatrix} \partial_X\\ \partial_Y \end{bmatrix}$ and $\nabla^{\perp}=\begin{bmatrix} -\partial_Y\\ \partial_X \end{bmatrix}$.

The flow is illustrated by the following figure.

\begin{figure}[h]
\includegraphics[scale=0.2]{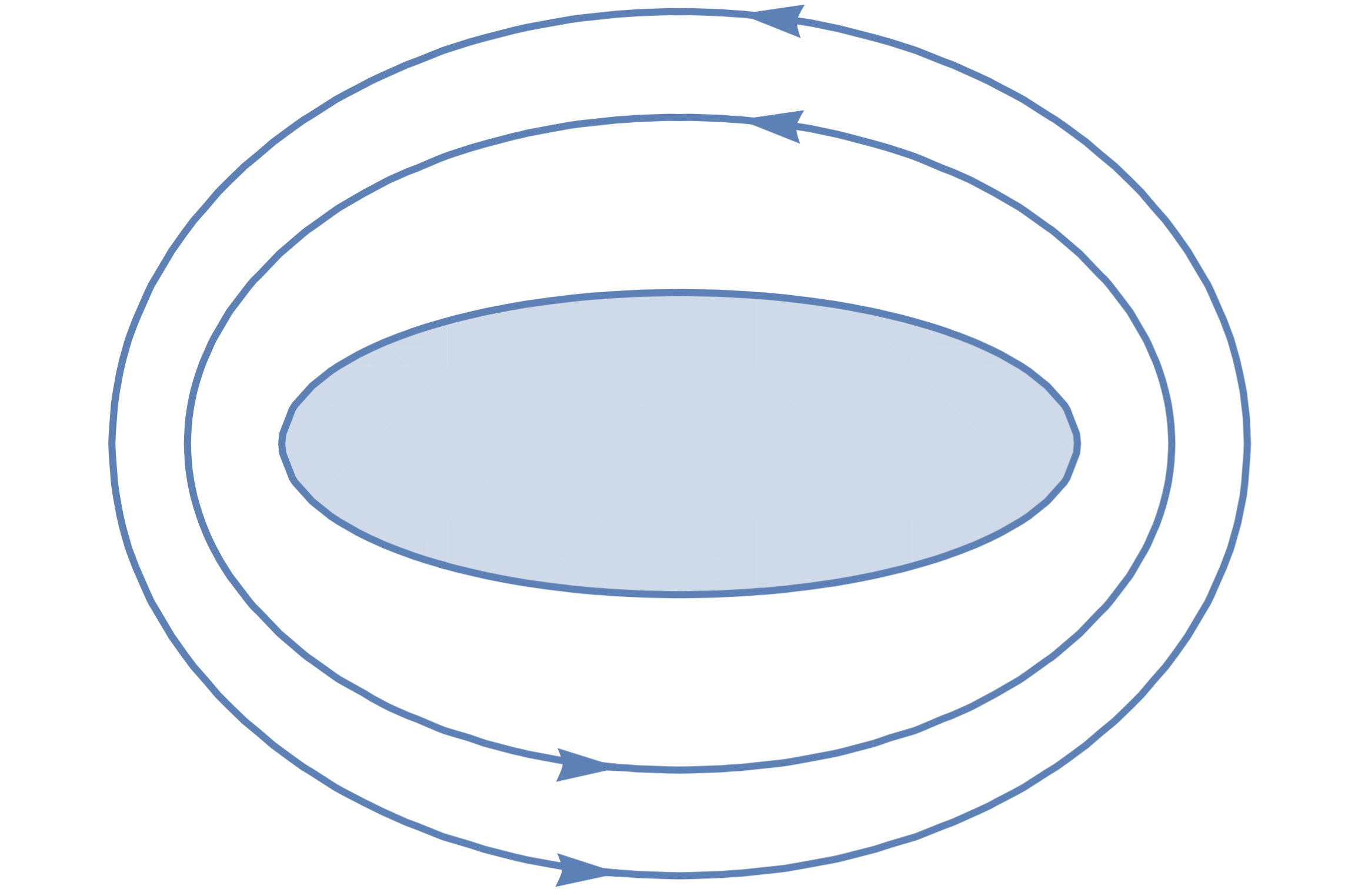}
\centering
\end{figure}

It can also be described by explicit formulas. First we can find $C$, $\Psi$ such that the parameters $A$, $B$ of the region $\mathcal{C}$ can be written as $A=Ccosh(\Psi)$, $B=Csinh(\Psi)$. Then the stream function $\psi_0(X,Y)$ of the flow is defined implicitly by the following equation
\begin{equation}\label{eq.stream}
    \frac{X^2}{C^2cosh^2(\psi_0+\Psi)}+\frac{Y^2}{C^2sinh^2(\psi_0+\Psi)}= 1.
\end{equation}

The velocity field of the flow is given by $u_0=\nabla^{\perp}\psi_0=\begin{bmatrix} -\partial_Y \psi_0\\ \partial_X\psi_0 \end{bmatrix}$.

A direct calculation shows that $\omega_0=\Delta \psi_0=0$, which proves that the flow is irrotational. The linearized Euler equation around this flow is

\begin{equation}\label{eq.perturbedeuler}
\left\{\begin{array}{ll}
     \partial_{t} \omega+\nabla^{\perp}\psi_0\cdot\nabla\omega=0. \qquad (X,Y)\in \mathbb{R}^2\backslash \mathcal{C}  \\
     \omega=\Delta \psi,\qquad \psi|_{\partial\mathcal{C}}=0,\qquad \psi|_{\infty}\sim const.
\end{array}\right.
\end{equation}
Here $\omega_0+\omega$, $\psi_0+\psi$ are the total vorticity and stream function of the perturbed flow. If $\omega_0\ne 0$, there should be a new term $\nabla^{\perp}\psi\cdot\nabla\omega_0$.

Since  $\omega_0=0$, this flow is a potential flow, so there exists $\varphi_{0}(X,Y)$ such that $u_0=\nabla \varphi_{0}$. This function is called the potential function of the flow.

In this paper we shall prove the following theorem which confirms that the flow governed by the linearzied Euler equation (\ref{eq.perturbedeuler}) is asymptotically stable under small perturbation.

\begin{theorem}\label{th.main} Assume that the perturbed flow governed by (\ref{eq.perturbedeuler}) has smooth initial data $\omega(0)$, $\psi(0)$. Then we have the following estimates

\begin{equation}\label{eq.main1}
    ||\psi(t)-\psi^{av}||_{L^2(\mathbb{R}^2\backslash \mathcal{C})}\lesssim (1+t)^{-2}||w\omega(0)||_{H^2(\mathbb{R}^2\backslash \mathcal{C})}
\end{equation}
     
\begin{equation}\label{eq.main2}
    ||u^{||}-u^{||}_{av}||_{L^2(\mathbb{R}^2\backslash \mathcal{C})}\lesssim (1+t)^{-1}||w\omega(0)||_{H^2(\mathbb{R}^2\backslash \mathcal{C})}
\end{equation}

\begin{equation}\label{eq.main3}
    ||u^{\perp}-u^{\perp}_{av}||_{L^2(\mathbb{R}^2\backslash \mathcal{C})}\lesssim (1+t)^{-2}||w\omega(0)||_{H^2(\mathbb{R}^2\backslash \mathcal{C})},
\end{equation}
where $u^{||}=\frac{u_{0}}{|u_0|}\cdot u$ and $u^{\perp}=\frac{\nabla\psi_0}{|\nabla\psi_0|}\cdot u$ are components of the velocity fields that are parallel (resp. perpendicular) to the flow line $\psi_0=c$. $w(X,Y)=(X^2+Y^2)^2$. $\psi^{av}$, $u^{||}_{av}$, $u^{\perp}_{av}$ are defined by following equations

\begin{equation}\label{eq.main4}
    \psi^{av}(X,Y)=r(\psi_0(X,Y))+s(\psi_0(X,Y)) cos(2\varphi_0(X,Y)),
\end{equation}

\begin{equation}\label{eq.main5}
    u^{||}_{av}(X,Y)=-W^{-1/2}(X,Y)[r'(\psi_0(X,Y))+s'(\psi_0(X,Y)) cos(2\varphi_0(X,Y))],
\end{equation}

\begin{equation}\label{eq.main6}
    u^{\perp}_{av}(X,Y)=-2W^{-1/2}(X,Y)s'(\psi_0(X,Y)) sin(2\varphi_0(X,Y)).
\end{equation}

Here $r$, $s$ are functions that can be calculated from $\varphi_0$, $\psi_0$, $\omega(0)$ using explicit formulas given in Lemma \ref{lem.solav} and $W(X,Y)$ is defined by

\begin{equation}
    W(X,Y)=\frac{C^2}{2}(cosh(2\psi_0(X,Y))-cos(2\varphi_0(X,Y))).
\end{equation}

%$W$ is an function in $X$, $Y$ implicitly defined in (\ref{eq.defofW}). 
% \begin{equation}
% \langle\psi\rangle(X,Y)=\frac{1}{2\pi}\int_0^{2\pi} \psi\Big(0,\ Ccos\varphi\  cosh\big(\psi_{0}(X,Y)+\Psi\big),-Csin\varphi\  sinh\big(\psi_{0}(X,Y)+\Psi\big)\Big) d\varphi    
% \end{equation}
% is the averge of the initial stream function along flow lines. ($\langle u^{||}\rangle(t,X,Y)$ is defined similarly.)
\end{theorem}
\begin{remark}
Unlike all previously established inviscid damping results, none of the velocity field components of this flow converges to 0. Hence the asymptotic flow lines are not parallel to unperturbed flow lines.
\end{remark}

\subsection{Some Physical Backgrounds} Now let's briefly explain the physical intuition behind above theorem. 

The asymptotic stability described in Theorem \ref{th.main} is often referred to as inviscid damping. A fluid exhibits inviscid damping if its the velocity fields and stream function converge in $L^2$ norm. This phenomenon is caused by the  vorticity mixing effect due to the difference of speeds between each layer of the flow. The mixing effect can be illustrated by following figures.

\begin{figure}[H]
\includegraphics[scale=0.9]{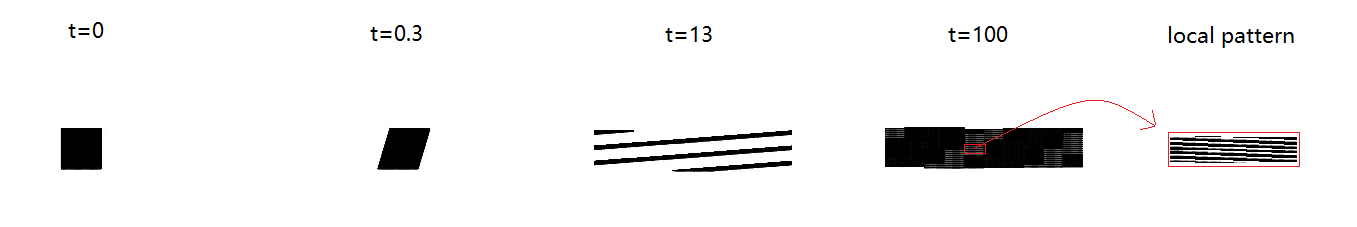}
\centering
\end{figure}

Above figures demonstrate how vorticity is transported on $\mathbb{T}\times[0,1]$. When mixing happens, since the upper fluid layer moves faster than lower fluid layer, the distribution of initial vorticity will be distorted, as a consequence the vorticity converges $L^2$ weakly but not strongly to its average. This convergence can't be strong or smooth (for example in $H^k$, $k>0$). If taking a closer look at the fourth figures, as shown in the fifth figure, we can observe that when $t\gg1$ the fluid oscillates wildly in the small scale, which prevents strong convergence. However, we have strong (but not smooth) convergence of the stream function and velocity field since they are obtained from the vorticity by applying an averaging operator (the inverse Laplacian).

The convergence of a quantity to its average is equivalent to that all of its Fourier modes other than the zero mode converge to 0. Hence mixing is easier to prove by taking the Fourier transform and we will apply this approach to prove Theorem \ref{th.main}.

The boundary of $\mathbb{R}^2\backslash \mathcal{C}$ does have a significant effect on the dynamics of the fluid, but in Theorem \ref{th.main} the convergence is in the $L^2$ average sense, so the boundary effect does not matter. 

In the previously well-understood case of circulation around a circular cylinder, all flows with circular flow lines are solutions of the steady Euler equation. In contrast, in the case of the elliptical cylinder the irrotational flow described by (\ref{eq.stream}) is the only solution of the steady Euler equation in which flow lines are ellipses, thus the asymptotic state of any flow cannot have ellipse flow lines. Therefore, the Euler equation has a tendency of driving the flow line away from ellipses, which explains why the asymptotic state of $u^{\perp}$ do not vanish. (If there's no such tendency the asymptotic flow lines should be ellpses.)

\subsection{Previous Works} There is a very rich literature on the inviscid damping phenomenon. Let us now discuss part of it that is related to this paper. 

\begin{enumerate}
    \item   The idea of introducing an auxiliary function $\phi$ to cancel the potential boundary terms from integration by parts essentially comes from \cite{Z1}. We will apply this argument in Section \ref{sec.fourier}. 
    \item   The proofs of linear and nonlinear inviscid damping of the circulation around a circular cylinder are very similar to that of the shear flow in a channel or the circulation in an annular domain. The linear stability of circulation in these two domains has been proved in \cite{WZZ}, \cite{Z3}. The very relevant linear stability problem of a point vortex has been studied in \cite{BCV}.
    \item   The boundary effects have been studied in \cite{Z1}, \cite{Z2}, \cite{J}.
    \item   The study of noninear inviscid damping was initiated in \cite{BM} for the Couette flow in an infinite channel, and extended in \cite{IJ0} to the finite channel case. The best results in the finite channel case now are \cite{IJ2}, \cite{MZ}, which apply to all linearly stable monotonic background flow. The similar infinite channel case seems to be easier but have not been proved yet. The very relevant nonlinear stability problem of a point vortex has been studied in \cite{IJ1}.
\end{enumerate}

\subsection{Notations}
If $f$, $g$ are functions of $x$, $y$, then 
\begin{equation}
    \frac{\partial(f,g)}{\partial(x,y)}=\begin{vmatrix}
f_x & f_y\\
g_x & g_y
\end{vmatrix}
\end{equation}

$\omega_0$ and $\psi_0$ are quantities of unperturbed flow. $\omega(0)=\omega(0,X,Y)$ and $\psi(0)=\psi(0,X,Y)$ are initial data. $\omega$ (resp. $\psi$) is total vorticity (resp. stream funcion) minus $\omega_0$ (resp. $\psi_0$).

{\bf Acknowledgments} The author would like to warmly thank his advisor Alexandru Ionescu and Hao Jia for the most helpful discussions.

\section{The Dynamics of Vorticity}
The linearized Euler equation (\ref{eq.perturbedeuler}) is exactly solvable. The goal of this section is to derive an explicit formula for the vorticity $\omega(t)$. We will do a change of variable in section \ref{sec.changevariables}, then the equation is solved by the method of characteristics in section \ref{sec.char}

\subsection{Change of Variables} \label{sec.changevariables} (\ref{eq.perturbedeuler}) is a first order PDE, so it can be solved by the method of characteristics. But this approach can be quite difficult since $\psi_0(X,Y)$ in (\ref{eq.perturbedeuler}) is a very complicated function in $X$, $Y$. To make the proof easier, we set $\psi_0$ as a new variable.

Since $\omega_0=0$, our flow is a potential flow. From the definition (\ref{eq.stream}) of the stream function and some lengthy calculations, we know that the potential $\varphi_0(X,Y)$ can be obtained by solving following equations

\begin{equation}
\begin{array}{cc}
     &X=Ccos\varphi_0\ cosh(\psi_0+\Psi)  \\
     &Y=-Csin\varphi_0\ sinh(\psi_0+\Psi) 
\end{array}
\end{equation}

This equation hints us to use $\varphi_0$, $\psi_0$ as new variables. So define new variables $x$, $y$ by
\begin{equation}\label{eq.changevar}
\begin{array}{cc}
     &X=Ccos(x)cosh(y+\Psi)  \\
     &Y=-Csin(x)sinh(y+\Psi) 
\end{array}
\end{equation}

The inverse transformation is 
\begin{equation}\label{eq.invtrans}
x=\varphi_{0}(X,Y)\qquad y=\psi_{0}(X,Y)
\end{equation}

Now we rewrite all quantities in the new coordinates. The domain $\mathbb{R}^2\backslash\mathcal{C}$ becomes $\mathbb{R}^2\backslash\mathcal{C}=\{(x,y):x\in [0,2\pi], y\in [0,+\infty)\}$.

$\nabla^{\perp}\psi_0\cdot\nabla\omega$ in (\ref{eq.perturbedeuler}) can be written as 

\begin{align*}
    \nabla^{\perp}\psi_0\cdot\nabla\omega=\frac{\partial(\psi_0,\omega)}{\partial(X,Y)}=\frac{\partial(y,\omega)}{\partial(X,Y)}=\frac{\partial(y,\omega)}{\partial(x,y)}/\frac{\partial(X,Y)}{\partial(x,y)}=-W^{-1}\partial_{x}\omega
\end{align*}

Here we have used 
\begin{align*}
    W=&\frac{\partial(X,Y)}{\partial(x,y)}=det
    \begin{bmatrix} X_x & X_y \\ Y_x & Y_y\end{bmatrix}=C^2\left(cosh^2(y+\Psi)-cos^2(x)\right)
    \\
    =&\frac{1}{2}C^2\left(cosh(2(y+\Psi))-cos(2x)\right)
\end{align*}
and $\frac{\partial(y,\omega)}{\partial(x,y)}=-\partial_x \omega$.

We may also verify that (\ref{eq.changevar}) is a conformal transform with determinant $W$, thus $\Delta_{x,y}=W\Delta_{X,Y}$. In what follows we will omit the subscript of $\Delta_{x,y}$ and $\Delta_{X,Y}$ for the ease of notation.

Thus (\ref{eq.perturbedeuler}) in the new variables reads
\begin{equation}\label{eq.eulervar}
\left\{\begin{array}{ll}
     \partial_{t} \omega-W^{-1}\partial_x\omega=0. \qquad(x,y)\in [0,2\pi]\times[0,+\infty)
     \\
     W\omega=\Delta \psi \qquad \psi|_{\partial\mathcal{C}}=0.
\end{array}\right.
\end{equation}

\begin{equation}\label{eq.defofW}
    W(x,y)=\frac{1}{2}C^2\left(cosh(2(y+\Psi))-cos(2x)\right),\ \ y\ge 0.
\end{equation}

\subsection{Method of Characteristics}\label{sec.char} In this section, we solve (\ref{eq.eulervar}) using the method of characteristics.

The characteristic equation of (\ref{eq.eulervar}) is 
\begin{align*}
    &W(x,y)dx=-dt
    \\
    &x|_{t=0}=a
\end{align*}

Integrating above equation gives
\begin{equation}\label{eq.solchar}
    -t=\frac{1}{2}C^2\left((x-a)cosh(2(y+\Psi))-\frac{1}{2}(sin(2x)-sin(2a))\right).
\end{equation}

$a$ can be view as a function of $t$, $x$, $y$ by solving above equation.

One can check by direct calculation that 
\begin{equation}\label{eq.omega}
    \omega(t,x,y)=\omega(0,a(t,x,y),y)
\end{equation}
is a solution to the first equation of (\ref{eq.eulervar}). Here $\omega(0,x,y)$ is the initial value of $\omega$.

$\psi$ is obtained by solving 
\begin{equation}
    \Delta \psi(t,x,y)=W(x,y)\omega(0,a(t,x,y),y)
\end{equation}

In section \ref{sec.main}, we shall do another change of variable $x,y,a\rightarrow \chi,y,\alpha$, so the following calculations of the partial derivatives are useful.

\begin{lemma} Let 
\begin{equation}\label{eq.chi}
    \chi(x,y)=x-\frac{sin(2x)}{2cosh(2(y+\Psi))}
\end{equation}

\begin{equation}\label{eq.alpha}
    \alpha(a,y)=a-\frac{sin(2a)}{2cosh(2(y+\Psi))}
\end{equation}

Then
\begin{equation}\label{eq.chialpha}
    \chi=\alpha- \frac{2t}{C^2cosh(2(y+\Psi))}
\end{equation}

\begin{equation}
    \chi_{x}(x,y)=1-\frac{cos(2x)}{cosh(2(y+\Psi))},\qquad \alpha_{a}(a,y)=1-\frac{cos(2a)}{cosh(2(y+\Psi))}.
\end{equation}

\begin{equation}\label{eq.W/chi}
    W(x,y)/\chi_{x}(x,y)=cosh(2(y+\Psi)).
\end{equation}

\end{lemma}

\begin{remark}
The fact that $W(x,y)/\chi_{x}(x,y)$ is independent of $x$ will be crucial in the later part of the proof. 
\end{remark}

\begin{proof} (\ref{eq.chialpha}) follows from (\ref{eq.solchar}). All others follows from direct calculations.
\end{proof}

\section{The Decay Estimates}
In this section we shall prove Theorem \ref{th.main}. In section \ref{sec.zeromode}, we consider the dynamics of the zero mode of the flow. %and the average operator $\langle\cdot\rangle$. 
Then in section \ref{sec.fourier}, we derive nice integral formulas (\ref{eq.int}), (\ref{eq.int2}), (\ref{eq.int1}) that relate $L^2$ norm of $\psi$ (resp. $u^{||}$, $u^{\perp}$) to the Fourier transform of $W\omega(0)$. Finally in section \ref{sec.main} we gives a proof of Theorem \ref{th.main} by applying integration by parts in (\ref{eq.int}), (\ref{eq.int2}), (\ref{eq.int1}). Working with the Fourier transforms of $\omega$ and $\psi$ is important here because mixing are easier to prove in frequency space.

\subsection{The Zero Mode}\label{sec.zeromode} In this section we show that it suffices to prove Theorem \ref{th.main} under an additional assumption $\langle W\omega(0)\rangle=0$.

In the new variable $x$, $y$, the domain $\mathbb{R}^2\backslash \mathcal{C}$ becomes $[0,2\pi]\times [0,+\infty)$. By (\ref{eq.changevar}), we have $\psi(X,Y)=\psi(-Csin(x)\  sinh(y+\Psi),Ccos(x)\  cosh(y+\Psi))$. Define a new function $\psi^{new}(x,y)=\psi(-Csin(x)\  sinh(y+\Psi),Ccos(x)\  cosh(y+\Psi))$, then $\psi(X,Y)=\psi^{new}(x,y)$. By abuse of notation we denote both of two functions by $\psi$.

Define 
\begin{equation}
\langle\psi\rangle(y)=\frac{1}{2\pi} \int_{[0,2\pi]} \psi(x,y) dx    
\end{equation}
Here the function $\psi$ denotes $\psi^{new}$, so

\begin{align*}
    &\frac{1}{2\pi} \int_{[0,2\pi]} \psi(x,y) dx=\frac{1}{2\pi} \int_{[0,2\pi]} \psi^{new}(x,y) dx 
    \\
    =&\frac{1}{2\pi} \int_{[0,2\pi]} \psi(-Csin(x)\  sinh(y+\Psi),Ccos(x)\  cosh(y+\Psi)) dx
\end{align*}
Notice that by (\ref{eq.invtrans}), $y=\psi_{0}(X,Y)$, so $\langle\psi\rangle$ is the average of $\psi$ along the initial flow lines.

An arbitrary initial data $\omega(0)$ can be decomposed into two parts $\omega(0)=\frac{\langle W\omega(0)\rangle}{\langle W\rangle}+(\omega(0)-\frac{\langle W\omega(0)\rangle}{\langle W\rangle})$. By the linearity of the equation the solution can also be decomposed correspondingly. 

Define $\omega^{av}(t,x,y)=\frac{\langle W\omega(0)\rangle(y)}{\langle W\rangle(y)}$ ($av$ stands for average). It is already a stationary solution to (\ref{eq.eulervar}) or (\ref{eq.perturbedeuler}) since it does not depend on $t$, $x$. If $\psi^{av}$ is a solution to 
\begin{equation}\label{eq.psiav}
\Delta \psi^{av}=W\omega^{av}    
\end{equation}
then $\psi^{av}$ obviously does not decay. 

Then the initial data of $(\underline{\omega},\underline{\psi})=(\omega-\omega^{av}, \psi-\psi^{av})$ satisfies

\begin{equation}\label{eq.initomegabar}
    \langle W\underline{\omega}(0) \rangle=\langle W\omega(0)\rangle-\left\langle W\frac{\langle W\omega(0)\rangle}{\langle W\rangle}\right\rangle=0
\end{equation}

% Apply average operator $\langle\cdot\rangle$ to $\Delta\psi=W\omega$ and $\Delta\widetilde{\psi}=W\widetilde{\omega}$ 

% \begin{equation}\label{eq.psiaverderiv1}
%     \partial_{yy}\langle\psi\rangle=\langle W\omega\rangle=\langle W\omega(0)\rangle.
% \end{equation}

% \begin{equation}\label{eq.psiaverderiv2}
%     \partial_{yy}\langle\widetilde{\psi}\rangle=\langle W\widetilde{\omega}\rangle=\left\langle W\frac{\langle W\omega(0)\rangle}{\langle W\rangle}\right\rangle=\langle W\omega(0)\rangle
% \end{equation}

So from now on we make the following assumption

\begin{equation}\label{eq.averW}
    \langle W\omega(0)\rangle=0.
\end{equation}

Note that $\langle W\omega(t)\rangle$ is a conservative quantity, so under above assumption $\langle W\omega(t)\rangle=0$ for any $t$. This conservation can be proved by multiplying the first equation in (\ref{eq.eulervar}) by $W$ and applying average operator to it. 

$\langle W\omega(t)\rangle=0$ implies that

\begin{equation}\label{eq.avergepsi}
    \langle\psi(t)\rangle=0,\qquad \forall t
\end{equation}
This can be proved by applying average operator $\langle\cdot\rangle$ to $\Delta\psi(t)= W\omega(t)$ which gives $\partial_{yy}\langle\psi(t)\rangle=0$, then the boundary conditions $\psi|_{y=0}=0$, $\psi|_{+\infty}=const$ will imply $\langle\psi(t)\rangle=0$. As a consequence of (\ref{eq.avergepsi}), the constant at infinity $const$ equals to 0 since $\langle\psi(t)\rangle|_{\infty}=0$.

\subsection{An Integration Formula}\label{sec.fourier} In this section we derive the desired integral formula. If directly solving $\psi$ in terms of $W\omega$ and taking $L^2$ norm, we get an bad expression of $||\psi||_{L^2}$. Integration by parts in this bad expression gives rise to a harmful non-vanishing boundary term. To avoid boundary terms, we run a similar argument as in \cite{Z1}. 

Let $\phi$ be the solution of the following Dirichlet problem,
\begin{equation}\label{eq.phi}
    \left\{
    \begin{array}{cc}
         \Delta \phi= \psi  \\
         \phi|_{\partial \mathcal{C}} =\phi|_{\infty} =0 
    \end{array}
    \right.
\end{equation}

Here we note that $\langle\psi\rangle=0$ ensures the existence of solution of above equation (without this condition we can only require that $\phi|_{\infty} =const$).

By $\Delta\psi=W\omega$ and integration by parts,

\begin{equation}
\begin{split}
     &\langle \psi, \psi \rangle= \langle \psi, \Delta \phi \rangle
     \\
     =&-\int \psi \partial_y\phi dx|_{y=0}+\int \partial_y\psi \phi dx|_{y=0}+\langle \Delta\psi, \phi \rangle
     \\
     =&\langle \Delta\psi, \phi \rangle=\langle W\omega, \phi \rangle
\end{split}
\end{equation}

Here the boundary terms from integration by parts vanish since $\psi|_{\partial\mathcal{C}}=\phi|_{\partial\mathcal{C}}=0$.

Define $\widetilde{\omega}(0,\alpha,y)=\omega(0,a(\alpha,y),y)cosh(2(y+\Psi))$, where the relation between variables $a$, $\alpha$ is given by (\ref{eq.alpha}). Then substitute (\ref{eq.omega}) in above equation

\begin{equation}\label{eq.psiL^2}
\begin{split}
    || \psi||^2_{L^2}=&\iint_{[0,2\pi]\times \mathbb{R}_+} \omega(0,a(t,x,y),y) W(x,y)\phi(x,y) dx dy    
    \\
    =&\iint_{[0,2\pi]\times \mathbb{R}_+} \widetilde{\omega}(0,\alpha,y) W(x,y)\phi(x,y)/cosh(2(y+\Psi)) dx dy 
\end{split}
\end{equation}

Change the integration variable from $x$, $y$ to $\chi$, $y$ according (\ref{eq.chi})

\begin{equation}
\begin{split}
    || \psi ||^2_{L^2}=&\iint_{[0,2\pi]\times \mathbb{R}_+} \widetilde{\omega}(0,\alpha,y)\phi(x,y) W(x,y)/(\chi_{x}(x,y)cosh(2(y+\Psi)) d\chi dy    
    \\
    =&\iint_{[0,2\pi]\times \mathbb{R}_+} \widetilde{\omega}\left(0,\chi+\frac{2t}{C^2cosh(2(y+\Psi))},y\right)\phi(x(\chi,y),y)  d\chi dy  
\end{split}
\end{equation}

Here in the last step we have applied (\ref{eq.chialpha}) and (\ref{eq.W/chi}). The fact that $W(x,y)/\chi_{x}(x,y)$ is independent of $x$ is important since the dependence on $x$ can radically destroy the integration by parts argument in the next section. (In this case do integration by parts may not necessarily offer decay in $t$.)

For any function $f(x,y)$, define

\begin{equation}
    \hat{f}_{k}(y)=\int_{[0,2\pi]} f(x,y) e^{-i kx} dx
\end{equation}

% Then apply the inversion formula for Fourier coefficient to $\phi$

% \begin{equation}\label{eq.phifourier}
%     \phi(x(t,a,y),y)=\sum_{k\ne 0}\widehat{\phi}_{k}(y) e^{i kx(t,a,y)}.
% \end{equation}

Applying the Plancherel identity in $x$ gives the desired integral formula  

\begin{equation}\label{eq.int}
    || \psi ||^2_{L^2}
    =\sum_{k\ne 0}\int_{\mathbb{R}_+} \widehat{\widetilde{\omega}(0)}_{k} (y)  e^{\frac{2i kt}{C^2cosh(2(y+\Psi))}}  \widehat{\phi}_{k}(y)  dy. 
\end{equation}

Here
\begin{equation}
    \widehat{\phi}_{k}(y)=\int_{[0,2\pi]} \phi(x(\chi,y),y) e^{ik\chi} d\chi.
\end{equation}
And $k\ne 0$ in the summation of (\ref{eq.int}) due to $\widehat{\widetilde{\omega}(0)}_{k} (y)|_{k=0}=0$. Now let's prove $\widehat{\widetilde{\omega}(0)}_{k} (y)|_{k=0}=0$,

\begin{align*}
     &\widehat{\widetilde{\omega}(0)}_{0} (y)=\int_{[0,2\pi]} \widetilde{\omega}(0,\alpha,y)d\alpha
     \\
     =&\int_{[0,2\pi]} \omega(0,a,y) cosh(2(y+\Psi))a_{\alpha} da=\int_{[0,2\pi]} \omega(0,a,y) W(a,y) da
     \\
     =&\langle W\omega\rangle=0
\end{align*}
Here in the last step we applied assumptions (\ref{eq.averW}).

A similar derivation gives
\begin{equation}\label{eq.int2}
    ||\partial_{x} \psi||^2_{L^2}=\sum_{k\ne 0}ik\int_{\mathbb{R}_+} \widehat{\chi^2_{x}\partial_{\chi}\widetilde{\omega}(0)}_{k} (y)  e^{\frac{2i kt}{C^2cosh(2(y+\Psi))}}  \widehat{\phi}_{k}(y)  dy. 
\end{equation}

\begin{equation}\label{eq.int1}
    ||\partial_{y}\psi||^2_{L^2}=\sum_{k\ne 0}\int_{\mathbb{R}_+} \partial_y\left(\widehat{\widetilde{\omega}(0)}_{k} (y)  e^{\frac{2i kt}{C^2cosh(2(y+\Psi))}}\right)  \partial_y\widehat{\phi}_{k}(y)  dy. 
\end{equation}

We now prove that 

\begin{equation}\label{eq.uparper}
    u^{||}=-W^{-1/2}\psi_y, \qquad u^{\perp}=W^{-1/2}\psi_x.
\end{equation}

Since $W^{-1/2}\lesssim 1$, the left hand side of (\ref{eq.int2}) and (\ref{eq.int1}) can bound $||u^{||}||$ and $||u^{\perp}||$.

It's easy to show that (\ref{eq.changevar}) is a conformal transformation, thus
\begin{align*}
    |\nabla_{X,Y}\psi_{0}|^2=|\nabla_{X,Y}\varphi_{0}|^2=\frac{\partial(\varphi_{0},\psi_{0})}{\partial(X,Y)}=\frac{\partial(x,y)}{\partial(X,Y)}=W^{-1}
\end{align*}

Thus 
\begin{align*}
    u^{\perp}=&\frac{\nabla\psi_0}{|\nabla\psi_0|}\cdot u=\nabla\psi_{0}\cdot \nabla^{\perp}\psi/|\nabla\psi_{0}|=\frac{\partial(\psi,\psi_{0})}{\partial(X,Y)}/W^{-1/2}
    \\
    =& \frac{\partial(\psi,y)}{\partial(x,y)}/\left(\frac{\partial(X,Y)}{\partial(x,y)}W^{-1/2}\right)
    \\
    =&W^{-1/2}\psi_x
\end{align*}

Since $u_0=\nabla \varphi_0$, ($\varphi_0$ is the potential function.)

\begin{align*}
    u^{||}=&\frac{u_{0}}{|u_0|}\cdot u=\frac{\nabla\varphi_0}{|\nabla\varphi_0|}\cdot \nabla^{\perp}\psi=\frac{\partial(\psi,\varphi_{0})}{\partial(X,Y)}/W^{-1/2}
    \\
    =& \frac{\partial(\psi,x)}{\partial(x,y)}/\left(\frac{\partial(X,Y)}{\partial(x,y)}W^{-1/2}\right)
    \\
    =&-W^{-1/2}\psi_y.
\end{align*}

Therefore we have proved (\ref{eq.uparper}).

\subsection{Proof of the Main Theorem}\label{sec.main}In this section we shall prove Theorem 1, using integration by parts in $y$ to get decay in $t$. 

\begin{proof}

We first prove Theorem \ref{th.main} under the assuming (\ref{eq.averW}), then this assumption will be removed at the end of the proof.

Let us first prove (\ref{eq.main1}). By (\ref{eq.int}),

\begin{equation}
    || \psi ||^2_{L^2}
    =\sum_{k\ne 0}\int_{\mathbb{R}_+} \widehat{\widetilde{\omega}(0)}_{k} (y)  \widehat{\phi}_{k}(y)  \left(\frac{C^2cosh^2(2(y+\Psi))}{2ikt\ sinh(2(y+\Psi))}\partial_y\right)^2 e^{\frac{2i kt}{C^2cosh(2(y+\Psi))}}  dy. 
\end{equation}

Do integration by parts twice in above equation with respect to $y$,

\begin{equation}
\begin{split}
    &|| \psi ||^2_{L^2}
    =-\sum_{k\ne 0}\left(\frac{\widehat{\widetilde{\omega}(0)}_{k} (y)  \widehat{\phi}_{k}(y) }{2ikt}\frac{C^2cosh^2(2(y+\Psi))}{sinh(2(y+\Psi))}e^{\frac{2i kt}{C^2cosh(2(y+\Psi))}}\right)\Bigg|_{y=0}
    \\
    +&\sum_{k\ne 0}\left(\frac{1}{(2ikt)^2}\frac{C^2cosh^2(2(y+\Psi))}{sinh(2(y+\Psi))}\partial_y\left(\frac{C^2cosh^2(2(y+\Psi))\widehat{\widetilde{\omega}(0)}_{k} (y)  \widehat{\phi}_{k}(y) }{sinh(2(y+\Psi))}\right)e^{\frac{2i kt}{C^2cosh(2(y+\Psi))}}\right)\Bigg|_{y=0}
    \\
    +&\sum_{k\ne 0}\int_{\mathbb{R}_+} \left[\left(\frac{C^2cosh^2(2(y+\Psi))}{2ikt\ sinh(2(y+\Psi))}\partial_y\right)^2\left(\widehat{\widetilde{\omega}(0)}_{k} (y)  \widehat{\phi}_{k}(y)\right)\right]   e^{\frac{2i kt}{C^2cosh(2(y+\Psi))}}  dy.     
\end{split}
\end{equation}

By (\ref{eq.phi}), $\phi|_{y=0}=0$, so the first boundary term in above integration by parts vanishes. Noticing that $cosh(2(y+\Psi)),\ sinh(2(y+\Psi)) =O_{\Psi}(e^{y})$, $|| \psi ||^2_{L^2}$ can be bounded by

\begin{equation}\label{eq.psilast}
\begin{split}
    || \psi ||^2_{L^2}
    =\underbrace{\sum_{k\ne 0}O_{\Psi}\left(\frac{\widehat{\omega(0)}_{k} (y)  \partial_y\widehat{\phi}_{k}(y)}{(kt)^2}\right)\Bigg|_{y=0}}_{(\text{\ref{eq.psilast}.I})}
    +\underbrace{\sum_{k\ne 0}\int_{\mathbb{R}_+} O_{\Psi}\left(\frac{e^{2y}}{(kt)^2}\partial_{yy}[\widehat{\widetilde{\omega}(0)}_{k} (y)  \widehat{\phi}_{k}(y)]\right) dy}_{(\text{\ref{eq.psilast}.II})}.     
\end{split}
\end{equation}

We need the following standard elliptic estimates.

\begin{lemma}\label{lem.ell}
Assume that $\phi$ is a solution to (\ref{eq.phi}) and $\langle\psi\rangle=0$, then we have

\begin{equation}
    ||\phi||_{H^2(\mathbb{R}^2\backslash\mathcal{C})}\lesssim ||\psi||_{L^2(\mathbb{R}^2\backslash\mathcal{C})}
\end{equation}
\end{lemma}

(\ref{eq.psilast}.I) can be bounded by

\begin{align*}
    &\sum_{k\ne 0}O_{\Psi}\left(\frac{\widehat{\omega(0)}_{k} (y)  \partial_y\widehat{\phi}_{k}(y)}{(kt)^2}\right)\Bigg|_{y=0}
    =\sum_{k\ne 0}\int_{\mathbb{R}_+}\partial_{y} O_{\Psi}\left(\frac{\widehat{\omega(0)}_{k} (y)  \partial_y\widehat{\phi}_{k}(y)}{(kt)^2}\right) dy
    \\
    =&\sum_{k\ne 0}\int_{\mathbb{R}_+} O_{\Psi}\left(\partial_{y}\widehat{\omega(0)}_{k} (y)  \partial_y\widehat{\phi}_{k}(y)\right) dy/(kt)^2
    +\sum_{k\ne 0}\int_{\mathbb{R}_+} O_{\Psi}\left(\widehat{\omega(0)}_{k} (y)  \partial_{yy}\widehat{\phi}_{k}(y)\right) dy/(kt)^2
    \\
    \lesssim& \sum_{k\ne 0} ||\partial_{y}\widehat{\omega(0)}_{k}||_{L^2_{y}}||\partial_y\widehat{\phi}_{k}||_{L^2_{y}} /(kt)^2
    +\sum_{k\ne 0} ||\widehat{\omega(0)}_{k}||_{L^2_{y}}||\partial_{yy}\widehat{\phi}_{k}||_{L^2_{y}} /(kt)^2
    \\
    \lesssim& \left|\left|||\partial_{y}\widehat{\omega(0)}_{k}||_{L^2_{y}}\right|\right|_{l^2_{k}}\  \left|\left| ||\partial_y\widehat{\phi}_{k}||_{L^2_{y}} /(kt)^2\right|\right|_{l^2_{k}}
    +\left|\left|||\widehat{\omega(0)}_{k}||_{L^2_{y}}\right|\right|_{l^2_{k}}\  \left|\left| ||\partial_{yy}\widehat{\phi}_{k}||_{L^2_{y}} /(kt)^2\right|\right|_{l^2_{k}} 
    \\
    \lesssim& ||\omega(0)||_{L^{2}_{\chi}H_{y}^2}  ||\phi||_{H^2} /t^2
    \\
    \lesssim& ||\omega(0)||_{H^2}  ||\psi||_{L^2} /t^2
\end{align*}

A closer inspection of above proof suggests that $||\omega(0)||_{H^2}$ can be replaced by $||\omega(0)||_{H^{-1}_{\chi}H_{y}^2}$. But finally after changing variables to $X$, $Y$, the best regularity one can hope is $||\omega(0)||_{H^2}$, so we do not pursue sharper estimates in above proof.

Here in the first equality we have used $f(0)=-\int^{+\infty}_{0} f'(y) dy$. The first and second inequality follows from Cauchy-Schwarz inequality. The last inequality follows from Lemma \ref{lem.ell}.

Using Cauchy-Schwatz inequality (\ref{eq.psilast}.II) can be bounded by

\begin{align*}
    &\sum_{k\ne 0}\left|\left|e^{2y}\widehat{\widetilde{\omega}(0)}_{k}\right|\right|_{H^2_{y}}  \left|\left|\widehat{\phi}_{k}\right|\right|_{H^2_{y}} /(kt)^2
    \\
    \lesssim& ||e^{2y}\widetilde{\omega}(0)||_{H^2}  ||\phi||_{H^2} /t^2
    \\
    \lesssim& ||e^{4y}\omega(0)||_{H^2}  ||\psi||_{L^2} /t^2
\end{align*}

Combining all above estimates, we have

\begin{equation}
    || \psi ||^2_{L^2}= \text{(\ref{eq.psilast}.I)}+\text{(\ref{eq.psilast}.II)}\lesssim ||e^{4y}\omega(0)||_{H^2}  ||\psi-\langle\psi\rangle||_{L^2} /t^2
\end{equation}

Thus 

\begin{equation}
    || \psi ||_{L_{x,y}^2}= \text{(\ref{eq.psilast}.I)}+\text{(\ref{eq.psilast}.II)}\lesssim ||e^{4y}\omega(0)||_{H^2} /t^2
\end{equation}

Transforming back to coordinates, $X$, $Y$, noticing that $e^{2y(X,Y)}\lesssim (X^2+Y^2)^2$, we have finished the proof of (\ref{eq.main1}) assuming (\ref{eq.averW}). 

The proof of (\ref{eq.main2}) (resp. (\ref{eq.main3})) assuming (\ref{eq.averW}) is similar. We just need to integrate by parts in (\ref{eq.int1}) (resp. (\ref{eq.int2})) for once (resp. twice).

To remove the assumption (\ref{eq.averW}), we need the following lemma,

\begin{lemma}\label{lem.solav}
The solution of (\ref{eq.psiav}) is given by
\begin{equation}\label{eq.lem3,1}
    \psi^{av}(x,y)=r(y)+s(y) cos(2x), 
\end{equation}

\begin{equation}
    r(y)=-\frac{1}{2} C^2\int^{\infty}_{0} min(\widetilde{y},y) cosh(2t) \omega^{av}(\widetilde{y})  d\widetilde{y}
\end{equation}

\begin{equation}
    s(y)=\frac{C^2}{4}\int^{\infty}_{y} sinh(2(y+\widetilde{y})) \omega^{av}(\widetilde{y}) d\widetilde{y}
\end{equation}

Here $\omega^{av}(y)=\frac{\langle W\omega(0)\rangle(y)}{\langle W\rangle(y)}$.

\end{lemma}
\begin{proof}
By separation of variables we may assume that the solution of (\ref{eq.psiav}) is of the form $\psi^{av}(x,y)=r(y)+s(y) cos(2x)$.

Substitute in (\ref{eq.psiav}) gives

\begin{equation}
    r''(y)+(s''(y)-4s(y)) cos(2x)=\frac{1}{2}C^2\omega^{av}(y) cosh(2y)-\frac{1}{2}C^2\omega^{av}(y) cos(2x).
\end{equation}

Thus we have derived following equations of $r(y)$, $s(y)$.

\begin{equation}\label{eq.r}
    r''(y)=\frac{1}{2}C^2\omega^{av}(y) cosh(2y)
\end{equation}

\begin{equation}\label{eq.s}
    s''(y)-4s(y)=-\frac{1}{2}C^2\omega^{av}(y)
\end{equation}

Since $\psi|_{\partial\mathcal{C}}=0$, $\psi|_{\infty}=const$ the boundary conditions for $r$ is $r(0)=0$, $r(+\infty)=const$. The boundary conditions for $s$ is $s(0)=s(+\infty)=0$.

The general solutions for (\ref{eq.r}) and (\ref{eq.s}) can be obtained using variation of constants method, with two unspecified constants in the results. These unspecified constants can be determined by above boundary conditions. After doing all these steps, we obtain the solution formulas of this lemma.

\end{proof}

Now we prove Theorem \ref{th.main} for a general solution without assumption (\ref{eq.averW}). Recall that $\underline{\omega}=\omega-\omega^{av}$ satisfies (\ref{eq.averW}) by (\ref{eq.initomegabar}). Thus by above argument we know that $||\underline{\psi}(t)||_{L^2}\lesssim (1+t)^{-2}||w\omega(0)||_{H^2}$. So we have

\begin{equation}
    ||\psi(t)-\psi^{av}||_{L^2}=||\underline{\psi}(t)||_{L^2}\lesssim  (1+t)^{-2}||w\omega(0)||_{H^2}
\end{equation}

Changing back to variables $X$, $Y$, we finish the proof of (\ref{eq.main1}).

By linearity of the equation and (\ref{eq.uparper}), the zero modes for $u^{||}$ and $u^{\perp}$ are

\begin{equation}
    u^{||}_{av}=W^{-1/2}\psi^{av}_y,\qquad u^{\perp}_{av}=-W^{-1/2}\psi^{av}_x
\end{equation}

Substituting (\ref{eq.lem3,1}) into above equations we obtain (\ref{eq.main5}) and (\ref{eq.main6}). (\ref{eq.main2}) and (\ref{eq.main3}) follows from the linearity of all equations.

We have thus finished the proof. 
\end{proof}

\end{document}